\documentclass[a4paper,12pt]{amsart}

\usepackage{epsfig}
\usepackage{amsthm,amsfonts}
\usepackage{amssymb,graphicx,color}
\usepackage[all]{xy}
\usepackage{verbatim}
\usepackage{hyperref}
\usepackage{enumitem}

\usepackage[a4paper,top=2.5cm,bottom=2.5cm,left=3cm,right=3cm]{geometry}

\newtheorem{theorem}{Theorem}[section]
\newtheorem*{theorem*}{Theorem}

\newtheorem{corollary}[theorem]{Corollary}

\newtheorem{definition}[theorem]{Definition}

\newcommand{\C}{\mathbb{C}}

\title[Note on Milnor numbers of irreducible germs ]{Note on Milnor numbers of irreducible germs }

\author{Zbigniew Jelonek}
\address{Zbigniew Jelonek -
Instytut Matematyczny, Polska Akademia Nauk, \'Sniadeckich 8, 00-656 Warszawa, Poland.}
\email{najelone@cyf-kr.edu.pl}

\begin{document}


\begin{abstract}
Let $(\bf {V,0})\subset (\C^n,0)$ be a germ of a complex hypersurface and let $f: (\C^n,0)\to(\C^n,0)$ be a germ of a finite holomorphic mapping. If germs   $(\bf {V,0})$ and ${\bf W}:=(F^{-1}(\bf{ V})),0)$ are irreducible and with isolated singularities,  then $$\mu(F^{-1}(\bf{ V}))\ge \mu(\bf {V}),$$ where $\mu$ denotes the Milnor number.
\end{abstract}
\maketitle

 \section{Introduction}

\medskip

This paper is devoted to proving the following  theorem:

\begin{theorem} Let $(\bf {V,0})\subset (\C^n,0)$ be a germ of a complex hypersurface and let $f: (\C^n,0)\to(\C^n,0)$ be a germ of a finite holomorphic mapping. If germs  germ $(\bf {V,0})$ and ${\bf W}:=(F^{-1}(\bf{ V})),0)$ are irreducible and with isolated singularities,  then $$\mu(F^{-1}(\bf{ V}))\ge \mu(\bf {V}),$$ where $\mu$ denotes the Milnor number.
\end{theorem}

As a Corollary we obtain the main result of \cite{tw}:

\begin{corollary} Let $(\bf {V,0})\subset (\C^n,0)$ be a germ of a complex hypersurface and let $f: (\C^n,0)\to(\C^n,0)$ be a germ of a finite holomorphic mapping. Then if the germ $(F^{-1}(\bf{ V})),0)$ taken with the reduced structure is smooth, then also the germ 
$({\bf V,0})$ is smooth.
\end{corollary}

\section{Main results}

We start with

\begin{definition}  
Let $(\bf {V,0})\subset (\C^n,0)$ be a germ of a complex hypersurface with an isolated singularity at $0$ and let ${\bf f}\in \mathcal{O}_0$ be a generator of the ideal $I(\bf {V,0})$ Then the Milnor number $\mu({\bf f})={\rm dim} \ \mathcal{O}_0/(\frac {\partial{\bf f}}{\partial {x_1}},..., \frac{ \partial{\bf f}}{\partial {x_n}}) $ does not depend on $\bf f$ but only on $(\bf {V,0}).$ We write $\mu(\bf {V,0}):=\mu(f).$ 
\end{definition}

It is easy to see that this definition is well-defined, i.e., the number $\mu$ does not depend on $f$ but only on $(\bf {V,0}).$ Now we can prove our main result:

\begin{theorem}\label{main-theorem} Let $(\bf {V,0})\subset (\C^n,0)$ be a germ of a complex hypersurface and let $f: (\C^n,0)\to(\C^n,0)$ be a germ of a finite holomorphic mapping. If germs  germ $(\bf {V,0})$ and ${\bf W}:=(F^{-1}(\bf{ V})),0)$ are irreducible and with isolated singularities,  then $$\mu(F^{-1}(\bf{ V}))\ge \mu(\bf {V}).$$
\end{theorem}

\begin{proof} 
We can assume that $n>1.$  Take $\epsilon >0$ so small that all spheres $S_\eta=\{ z: |z|=\eta\}$ are transversal to $V$ for $0<\eta\le\epsilon$. Let $g$ be a reduced equation for $V.$ We can assume that for $|c|<\delta$ the fibers $\{f(z)=c\}$ are transversal to $S_\epsilon$ and they  are Milnor fibers for $g$ in  $B(0,\epsilon)$ (see \cite{mil}). Take a ball $B'$ so small that $F(B')\subset B(0,\epsilon).$ There is an $\eta_0>0$ so small that $F^{-1}(B(0,\eta_0))\subset B'.$ We can take $\delta_0$ so small that for every $C$ with $|c|\le\delta_0$ fibers $g=c$ are transversal to spheres $S_\eta$ for all $\eta_0\le \eta \le \epsilon.$ 

Let $h$ be a reduced (local) equation of $W=F^{-1}(V).$ By the holomorphic Nullstellensatz we have $g\circ F=h^r$ where $r>0$ is a natural number. Hence $(g-c)(F)=\prod^r_{i=1} (h-\alpha_ic^{1/r})$, where $\alpha_i$ are all roots of order $r$ from $1.$ In particular we see that the fiber $\{h=\alpha_ic^{1/r}\}$ is transformed onto fiber $\{g=c\}.$ Let $A_c$ be a Milnor fiber $\{g=c\}\cap B(0,\eta_0)$ and let $B_c=F^{-1}(A_c)\cap \{h=c^{1/r}\}.$ Since the mapping $F: B_c\to A_c$ is finite, we have that $F_*: H_k(B_c)\to H_k(A_c)$ is surjective for every $k$
(see \cite{surj}).

By the Milnor theorem rank $H_{n-1}(A_c)=\mu({\bf V})=s\ge  1,$ where $\mu({\bf V})$ is the Milnor number of $g.$ In particular there are $n-1$ cycles $\alpha_1,...,\alpha_s\subset A_c$, such that $[\alpha_i]$ are generators of $H_{n-1}(A_c).$ Let $\beta_i; \ i=1,...,s$ be  cycles in $B_c$ such that $F(\beta_i)=\alpha_i.$ Let $A_c'=
\{g=c\}\cap B(0,\epsilon).$ By our assumptions the fiber $A_c$ is a deformation retract of $A_c'.$ Hence the cycles $\alpha_i; \ i=1,...,s$ are generators of  $H_{n-1}(A_c')$, too.  Let $B_c'=\{h=c^{1/r})\}\cap B'$ be a Milnor fiber of $h$ in $B'.$ 
Consider  the mapping $$F_* : H_{n-1}(B_c')\to H_{n-1}(A_c').$$ We have $F_*([\beta_i])= [{\alpha}_i] $. Hence rank $ H_{n-1}(B_c')\ge$ rank $H_{n-1}(A_c')$, i.e., $$\mu(F^{-1}(\bf{ V}))\ge \mu(\bf {V}).$$
\end{proof}

\begin{corollary}\label{cormain-theorem} Let $(\bf {V,0})\subset (\C^n,0)$ be a germ of a complex hypersurface and let $f: (\C^n,0)\to(\C^n,0)$ be a germ of a finite holomorphic mapping. Then if the germ $(F^{-1}(\bf{ V})),0)$ taken with the reduced structure is smooth, then also the germ 
$({\bf V,0})$ is smooth.
\end{corollary}

\begin{proof}
 We proceed by induction. Note that if  ${\bf V}$ has an isolated singularity, then it  is smooth if and only if $\mu({\bf V})=0$. In particular in this case  our Corollary follows directly from Theorem \ref{main-theorem}. 
 In particular we get the thesis if  $n=2.$  

Now assume that $n>2$ and ${\bf S}:=Sing(\bf {V,0})$ has dimension greater than $0.$ Let $\mathcal{L}$ be a linear system of hyperplanes in $\C^n.$ On $V$ let us consider the induced linear system $\mathcal{L}_V$. 
Then a generic member $L_V$ of $\mathcal{L}_V$ is singular, because $S\cap L\not=\emptyset$ (see Lemma 2.5 in \cite{j-b}). Moreover $F^*(L_V)$ (and $F^{-1}(L)$) are  smooth because the system $F^*(\mathcal{L}_V)$ has no base points on smooth $W=F^{-1}(V)\subset F^{-1}(L)$ (the system $F^*(\mathcal{L})$ has no base points on $\C^n$). Since dim $L\cap V<$ dim $V$, this is a contradiction.
\end{proof}
\medskip

{\it Acknowledgement}. The author is grateful to prof. Maciej Denkowski and prof. Karolina Zaj\c ac for helpful conversations.

\end{document}